\documentclass{article}
\usepackage{maa-monthly}

\usepackage{graphics}
\usepackage[T1]{fontenc}
\usepackage{amsthm}
\usepackage{amsbsy,amsmath,amssymb,amscd,amsfonts}
\usepackage{hyperref}
\usepackage{url}
\usepackage{booktabs}
\usepackage{nicefrac}
\usepackage{microtype}
\usepackage{graphicx,float,latexsym,color}
\usepackage[font={small,it}]{caption}
\usepackage{subcaption}
\usepackage{makecell}

\usepackage[dvipsnames]{xcolor}
\newtheorem{theorem}{Theorem}

\newtheorem{corollary}{Corollary}

\theoremstyle{definition}

\newtheorem{remark}{Remark}
\hypersetup{
   pdftoolbar=true,   
    pdfmenubar=true,    
    pdffitwindow=false,  
    pdfstartview={FitH},  
    colorlinks=true,     
    linkcolor=OliveGreen,  
    citecolor=blue, 
    filecolor=black,  
    urlcolor=red 
}

\usepackage{lineno}

\begin{document}

\title{New Properties of Triangular Orbits\\in Elliptic Billiards}
\markright{Triangular Orbits in Elliptic Billiards}
\date{July, 2020}

%\author{Blind Copy}

 \author{Ronaldo Garcia, Dan Reznik, and Jair Koiller}

\maketitle

%%% BEGIN \input{001_abstract}
\begin{abstract}
New invariants in the one-dimensional family of 3-periodic orbits in the elliptic billiard were introduced by the authors in "Can the Elliptic Billiard Still Surprise Us?"  (2020) \textit{Math. Intelligencer} 42(1): 6--17, some of which were generalized to $N>3$. Invariants mentioned there included ratios of radii and/or areas, sum of angle cosines, and a special stationary circle. Here we present some of the proofs omitted there as well as a few new related facts. \\
\end{abstract}
%%% END \input{001_abstract}

\section{Introduction.}
\label{sec:intro}
%%% BEGIN \input{002_introduction}
The elliptic billiard is a particle moving with constant velocity in the interior of an ellipse, undergoing elastic collisions against its boundary \cite{rozikov2018,sergei91}; see Figure~\ref{fig:billiard-trajectories}. It satisfies two integrals of motion: (i) energy: constant velocity and elastic collisions, and (ii) Joachimsthal's: all trajectory segments are tangent to a virtual, confocal ellipse known as the {\em caustic} \cite{sergei91}. The confocal pair renders the elliptic billiard a special case of {\em Poncelet's porism} \cite{dragovic11}: if one $N$-periodic {\em orbit} (a closed trajectory) can be found departing from some boundary point, any other such point can initiate such an orbit, i.e., a 1-dimensional {\em family} of $N$-periodic orbits exists. Integrability produces a first magical consequence: their perimeter is invariant \cite{sergei91}. 

\begin{figure}[H]
    \centering
    \includegraphics[width=.8\textwidth]{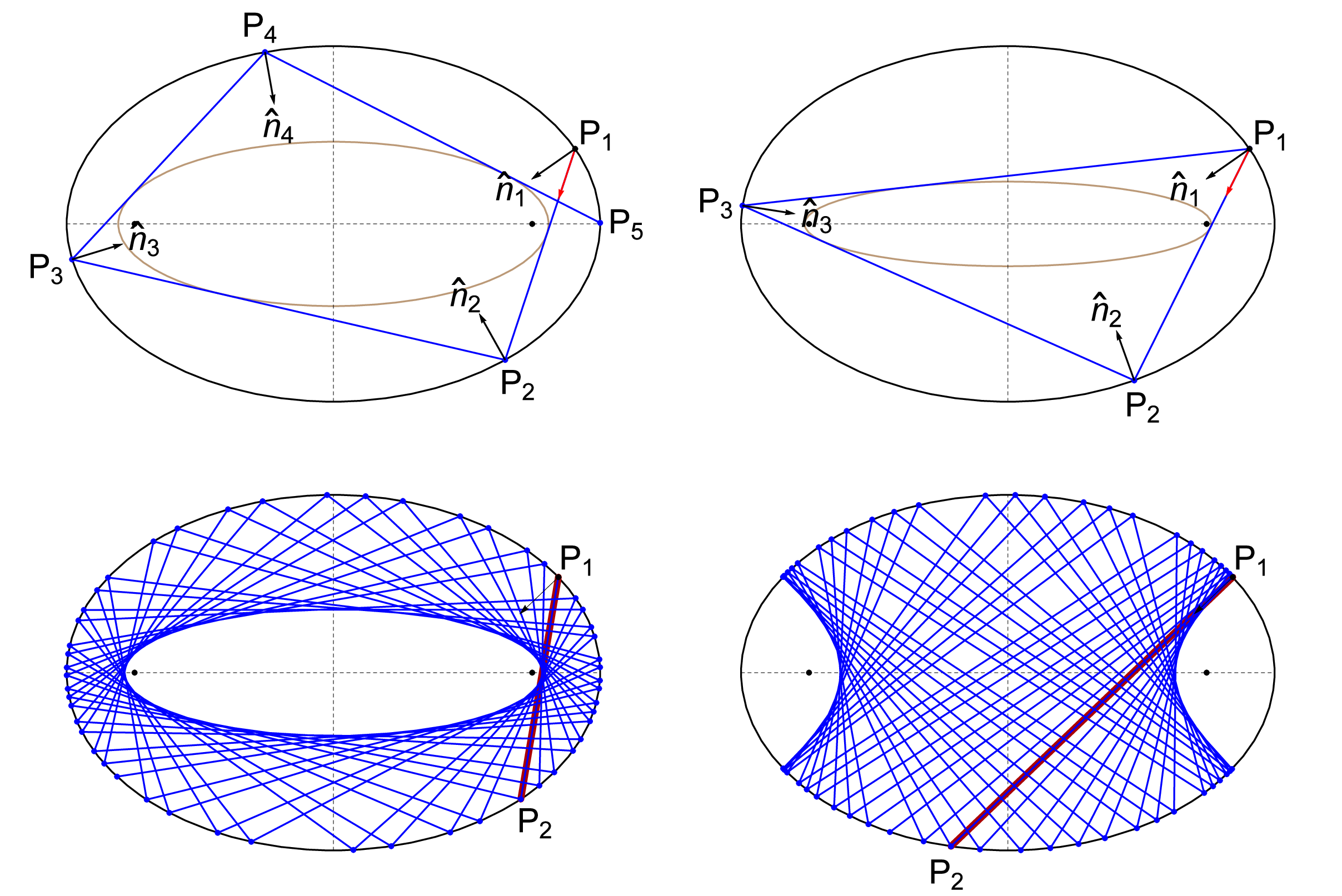}
    \caption{Trajectory regimes in an elliptic billiard. \textbf{Top left}: The first four segments of a trajectory departing at $P_1$ and moving toward $P_2$, bouncing at $P_i, i=2,3,4$. At each bounce the normal $\hat{n}_i$ bisects incoming and outgoing segments. Joachimsthal's integral \cite{sergei91} means all segments are tangent to a confocal {\em caustic} (brown). \textbf{Top right}: All 3-periodic orbits are tangent to a confocal caustic (brown). \textbf{Bottom}: The first 50 segments of a non-periodic trajectory starting at $P_1$ and directed toward $P_2$. Segments are tangent to a confocal ellipse (left) or hyperbola (right). The former (respectively, latter) occurs if $P_1P_2$ passes outside (respectively, between) the elliptic billiard's foci (black dots).}
    \label{fig:billiard-trajectories}
\end{figure}

Here we focus on invariants of 3-periodic orbits; see Figure~\ref{fig:three-orbits-proof}. Because these  are triangles, we make use of an array of classic properties. In \cite{reznik2019-intelligencer,reznik2020-loci} we analyzed the loci of {\em triangle centers} \cite{kimberling97-major-centers} such as the incenter, barycenter, etc. These yield a smorgasbord of algebraic curves: circles, ellipses, quartics, and higher order; see \cite{reznik2019-locus-gallery}.

One early observation was that the locus of the incenter (where bisectors meet) is an ellipse \cite[PL\#01]{reznik2020-playlist-proofs}. Proofs soon followed for the ellipticity of the incenter \cite{olga14}, barycenter \cite{sergei07_grid} and circumcenter \cite{corentin19,garcia2019-ellipses}, and more recently \cite{reznik2020-loci} for 29 out of the first 100 entries in Kimberling's copious Encyclopedia of Triangle Centers (ETC) \cite{etc}, where centers are identified as $X_i$, e.g., $X_1,X_2,X_3$ for incenter, barycenter, circumcenter.

\begin{figure}[H]
    \centering
    \includegraphics[width=.5\textwidth]{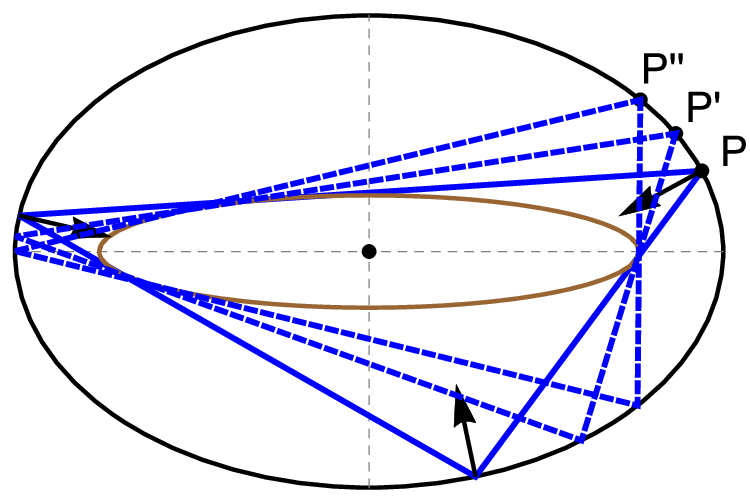}
    \caption{Three members (blue) of the one-dimensional family of 3-periodic orbits. These are triangles inscribed in an ellipse, whose vertices are bisected by the local normals. Joachimsthal's integral \cite{sergei91} prescribes that all trajectory segments (closed or not) are tangent to a confocal caustic (brown), i.e., this is a special case of Poncelet's porism \cite{dragovic88}. Moreover, the entire family of 3-periodics is tangent to that {\em same} caustic!}
    \label{fig:three-orbits-proof}
\end{figure}

Another interesting observation was that the {\em Mittenpunkt} $X_9$ (where lines drawn from each excenter through sides' midpoints concur \cite[Mittenpunkt]{mw}) is stationary at the billiard center \cite{reznik2019-intelligencer}; see \cite[PL\#02]{reznik2020-playlist-proofs}. \\

\subsection{Main Results.}

\begin{itemize}
    \item Theorem~\ref{thm:rovR}: The 3-periodic family conserves the ratio of inradius-to-circumradius, resulting in several corollaries;
    \item Theorem~\ref{thm:area-ratio-outer-inner}: Invariant area ratio between the 3-periodic orbits' excentral and orthic triangles.
    \item Theorem~\ref{thm:delta-power}: The power of the circumcircle of 3-periodic orbits with respect to the center of the elliptic billiard is invariant.
    \item Theorem~\ref{thm:cosine-circle}: 
    The cosine circle of the excentral triangle to 3-periodic orbits is stationary, centered on the Mittenpunkt, and exterior to the elliptic billiard.
    \item Theorem~\ref{thm:rovR-explicit}: An expression is derived for the invariant inradius-to-circumradius ratio in terms of the two classic invariants of the elliptic billiard (perimeter and Joachmisthal's constant).
\end{itemize}

\bigskip
\subsection{Related Work.}
Examples of triangle families under certain constraints which have been studied include the following: (i) Poristic triangles \cite{Gallatly1913}, with fixed incircle and circumcircle. These by definition   preserve  $r/R$ (an invariant shared with 3-periodics in the elliptic billiard). First studied by Chapple in 1761, then Euler, a modern treatment is given in \cite{Weaver1924,Weaver1933} and \cite{Murnaghan1925}. More recently, several of its triangle centers have been shown to produce circular and pointwise loci \cite{Odenhal2011}, and to be related to the billiard 3-periodics under a similarity transform \cite{garcia2020-poristics}. (ii) Common incircle and centroid: vertices lie on a conic \cite{Pamfilos2011}. (iii)
Triangles with sides tangent to a circle \cite{Nikolina-families2012}. (iv) Triangles associated with two lines and a point not on them \cite{Sliepcevic2013}. Also related is the family of rectangles inscribed in smooth curves \cite{schwartz2018-rectangles}. The locus of centroids of Poncelet polygons is studied in \cite{schwartz2016-com}, whereas that of the circumcenter is studied in \cite{ana2020}.

\bigskip
\subsection{Outline of the article.} In Section~\ref{sec:prelim} we introduce preliminaries. Section~\ref{sec:rovr} contains proofs for invariant inradius-to-circumradius ratio and corollaries. Section~\ref{sec:cosine-circle} describes a stationary circle associated with 3-periodic orbits. Generalizations for $N>3$ appear in Section~\ref{sec:gener}. Section~\ref{sec:experimental} presents experimental results which served as a guide to the theorems. A table of videos illustrating some of the phenomena discussed herein is provided in Appendix~\ref{app:videos} and Appendix~\ref{app:orbit-vertices} provides explicit expressions for 3-periodic orbit vertices.
%%% END \input{002_introduction}

\section{Preliminaries: Classic Invariants.}
\label{sec:prelim}
%%% BEGIN \input{004_preliminaries}
Let the boundary of the elliptic billiard satisfy

\begin{equation}
\label{eqn:billiard-f}
f(x,y)=\left(\frac{x}{a}\right)^2+\left(\frac{y}{b}\right)^2=1,\;\;\;a>b.
\end{equation}

Joachimsthal's integral implies that every trajectory segment is tangent to the caustic \cite{sergei91}. Equivalently, a positive quantity $\gamma$ (called $J$ in \cite{akopyan2020-invariants,bialy2020-invariants}) remains invariant, whether the trajectory is closed or not, at every bounce point $P_i$:

\begin{equation}
 \gamma=\frac{1}{2}\hat{v}.\nabla{f_i}=\frac{1}{2}|\nabla{f_i}|\cos\alpha,
 \label{eqn:joachim}
\end{equation}

\noindent where $\hat{v}$ is the unit incoming (or outgoing) velocity vector, and

\begin{equation*}
\nabla{f_i}=2\left(\frac{x_i}{a^2}\,,\frac{y_i}{b^2}\right).
\label{eqn:fnable}
\end{equation*}

Consider a starting point $P_1=(x_1,y_1)$ on the boundary of the elliptic billiard. The the exit angle $\alpha$ (measured with respect to the normal at $P_1$) required for the trajectory to close after 3 bounces is given by \cite{garcia2019-ellipses}: 

\begin{equation}
\cos{\alpha}={\frac {a^2 b \, \sqrt {2 \delta-{a}^{2}-{b}^{2}}}{{c}^{2}\sqrt {{a}^{4}-{c}^{2} x_1^{2}}}}, \;\; c^2=a^2-b^2,\;\; \delta=\sqrt{a^4-a^2b^2+b^4}\,.
\label{eqn:cosalpha}
\end{equation}

\noindent Note: Regarding the quantity $\delta$, see Theorem~\ref{thm:power_delta} (below) for a nice geometric interpretation. \\

Choosing $P_1=(a,0)$, we derive $\gamma$ based on equations~\eqref{eqn:joachim} and \eqref{eqn:cosalpha}:

\begin{equation}
\gamma=\frac{\sqrt{2\delta-a^2-b^2}}{c^2}.
\end{equation}

\noindent When $a=b$, $\gamma=\sqrt{3}/2$ and when $a/b{\to}\infty$, $\gamma{\to}0$. \noindent Using explicit expressions for the orbit vertices \cite{garcia2019-ellipses} (reproduced in Appendix~\ref{app:orbit-vertices}), we derive the perimeter $L=s_1+s_2+s_3$: 

\begin{equation}
L=2(\delta+a^2+b^2)\gamma.
\label{eqn:perimeter}
\end{equation}

\noindent Incidentally, quantity $\delta$ has the following geometric interpretation: it is the power of $X_9$ with respect to the orbit's circumcircle.
%%% END \input{004_preliminaries}

\section{Invariance of Inradius-to-Circumradius Ratio.}
\label{sec:rovr}
%%% BEGIN \input{005_invariant_ratio}
Referring to  Figure~\ref{fig:radii}, let $r$, $R$, and $r_9$ denote the radii of the incircle, circumcircle, and 9-point circle of a 3-periodic orbit, respectively. These are centered on $X_1$, $X_3$, and $X_5$, respectively \cite{coxeter67}. The Mittenpunkt $X_9$, is stationary at the elliptic billiard center over the 3-periodic orbit family \cite{reznik2019-intelligencer}.

\begin{figure}[H]
    \centering
    \includegraphics[width=.75\textwidth]{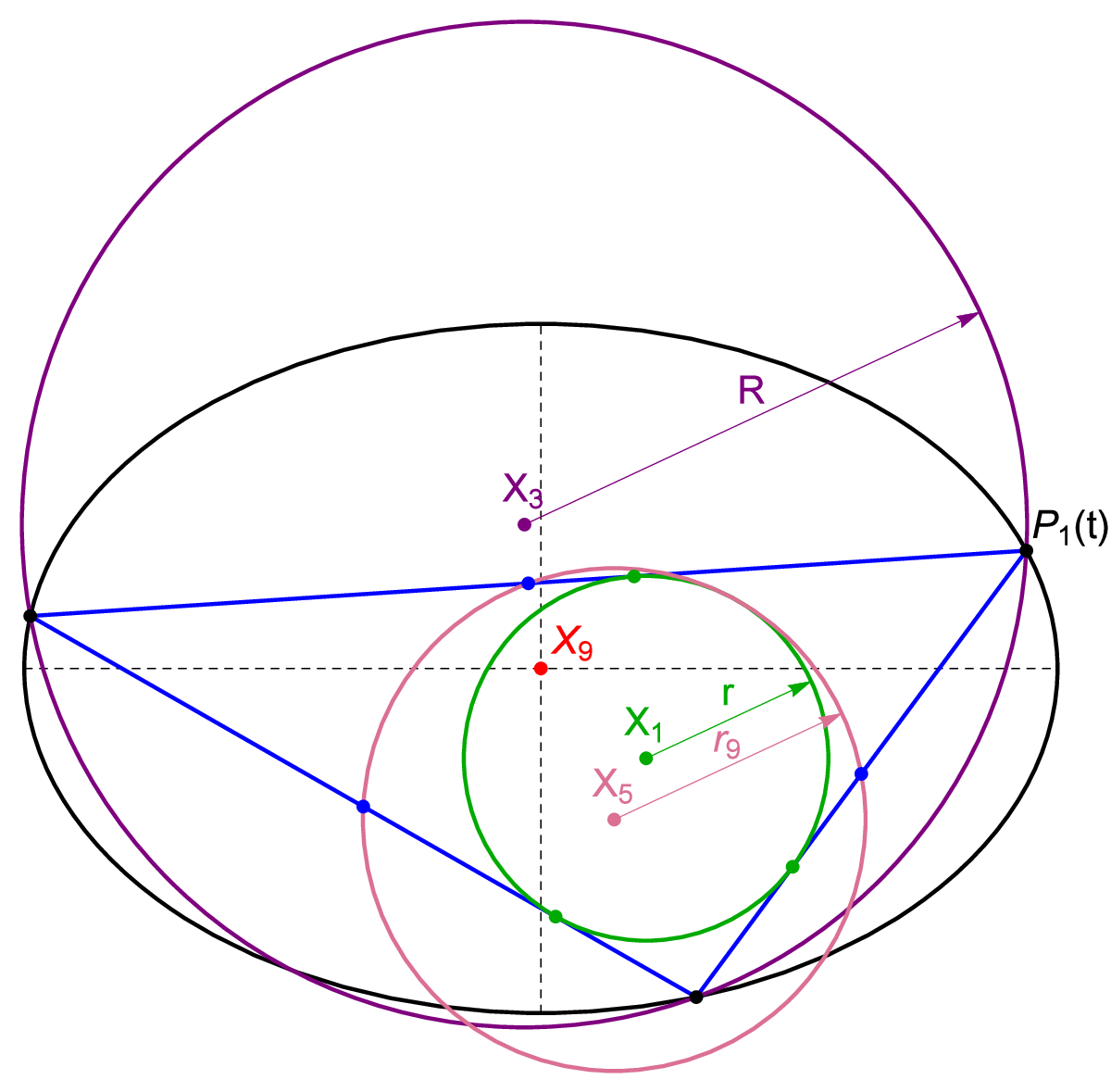}
    \caption{A 3-periodic orbit (blue), a starting vertex $P_1$, the incircle (green), circumcircle (purple), and 9-point circle (pink), whose centers are $X_1$, $X_3$, and $X_5$, and radii are the inradius $r$, circumradius $R$, and 9-point circle radius $r_9$. The Mittenpunkt $X_9$ is stationary at the billiard center \cite{reznik2019-intelligencer}.}
    \label{fig:radii}
\end{figure}

\noindent We now prove a result announced in  \cite{reznik2019-intelligencer}:

\begin{theorem}
\label{thm:rovR}
$r/R$ is invariant over the 3-periodic orbit family and given by
\begin{equation}
\label{eqn:rovR}
\frac{r}{R}=\frac{2 (\delta-b^2)(a^2-\delta)}{(a^2-b^2)^2}.
\end{equation}
\end{theorem}

\begin{proof}
Let $r$ and $R$ be the radius of the incircle and circumcircle, respectively. For any triangle \cite{coxeter67} we have

\begin{equation*}
 rR=\frac{s_1s_2s_3}{2 L}, 
\end{equation*}

\noindent where $L=s_1+s_2+s_3$ is the perimeter, constant for 3-periodic orbits; see equation \eqref{eqn:perimeter}. Therefore,

\begin{equation}
\frac{r}{R}=\frac{1}{2L} \frac{s_1s_2s_3}{R^2}\cdot
\label{eqn:rovR-cas}
\end{equation}

Next, with $P_1=(a,0)$, obtain a {\em candidate} expression for $r/R$. This yields \eqref{eqn:rovR} exactly. Using explicit expressions for orbit vertices (see Appendix~\ref{app:orbit-vertices}), derive an expression for the square of the right-hand side of \eqref{eqn:rovR-cas} as a function of $x_1$ and subtract from it the square of \eqref{eqn:rovR}. It can be shown $\left(s_1s_2s_3/R^2\right)^2$ is rational on $x_1$ \cite{reznik2020-loci}. For simplification, use $R=s_1 s_2 s_3/(4A)$, where $A$ is the triangle area. With a computer algebra system (CAS), show said difference is identically zero for all $x_1\in(-a,a)$.
\end{proof}

Though the ratio $r/R$ is invariant, the maxima and minima of $r$ and $R$ occur at the isosceles configurations of the 3-periodic orbits; see Figure~\ref{fig:rR-min-max}. 

\begin{figure}[H]
    \centering
    \includegraphics[trim={0 0.5cm 0 0},clip,width=\textwidth]{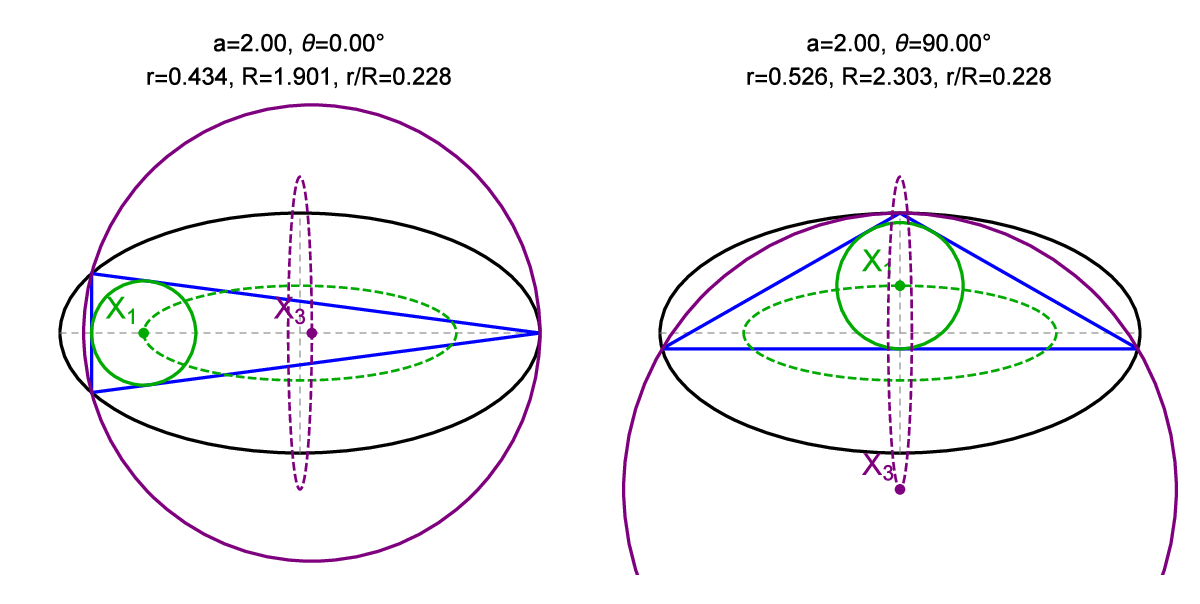}
    \caption{\textbf{Left}: When the 3-periodic orbit (blue) is a sideways isosceles triangle (one of its vertices   at either horizontal vertex of the elliptic billiard), $r$ and $R$ are minimal. The incircle and circumcircle are shown (green and purple, respectively), as are the loci of the incenter and circumcenter (dashed green and dashed purple, respectively). \textbf{Right}: When the orbit is an upright isosceles triangle (one vertex   the top or bottom vertex of the elliptic billiard), $r$ and $R$ are maximal.}
    \label{fig:rR-min-max}
\end{figure}

%\begin{proof}
%For any triangle we have the following relations.
%	\[rR= \frac{s_as_bs_c }{2(s_a +s_b +s_c	)}\]
%	\[R=\frac{r}{\cos A+\cos B+\cos C-1}=\sqrt{\frac{s_a^2+s_b^2+s_c^2}{8(1+  \cos A \cos B\cos C) }}\]
%\[ \frac{r}{R}= \cos A+\cos B+\cos C-1 \]
%\end{proof}

Note: The proof above and a few proofs below were obtained with the  assistence of   a Computer Algebra System. Indeed, expressions deriving from those in Appendix~\ref{app:orbit-vertices} quickly become unwieldy. We encourage the interested reader to produce simpler algebraic or, preferably, synthetic alternatives.

The three relations below hold for any triangle \cite{johnson29}. These will be used in subsequent corollaries.

\begin{eqnarray}
\sum_{i=1}^{3}{\cos\theta_i}&=&1+\frac{r}{R} \label{eqn:sum-cos} \\
\prod_{i=1}^{3}{|\cos\theta_i'|}&=&\frac{r}{4R} \label{eqn:exc-prod-cos} \\
\frac{A}{A'}&=&\frac{r}{2R}
\label{eqn:area-ratio}
\end{eqnarray}

\noindent where $\theta_i$ are the angles internal to the orbit, $\theta_i'$ are those of the excentral triangle (opposite to orbit $\theta_i$), and $A$ (respectively, $A'$) is the area of the orbit (respectively, excentral triangle); see Figure~\ref{fig:orbit-outer-inner}.

\begin{corollary}
The sum of the orbit cosines, the product of excentral cosines, and the ratio of excentral-to-orbit areas are all constant.
\end{corollary}

Note that the absolute value in \eqref{eqn:exc-prod-cos} can be dropped as the excentral triangle is always acute \cite{coxeter67}.

In \cite{reznik2019-intelligencer} we reported experiments that showed that the left-hand sides of \eqref{eqn:sum-cos} and \eqref{eqn:exc-prod-cos} were constant for all $N$-periodic orbits, whereas \eqref{eqn:area-ratio} was constant for odd $N$ only. These generalizations were subsequently proved \cite{akopyan2020-invariants,bialy2020-invariants}.

\begin{figure}[H]
    \centering
    \includegraphics[width=.6\textwidth]{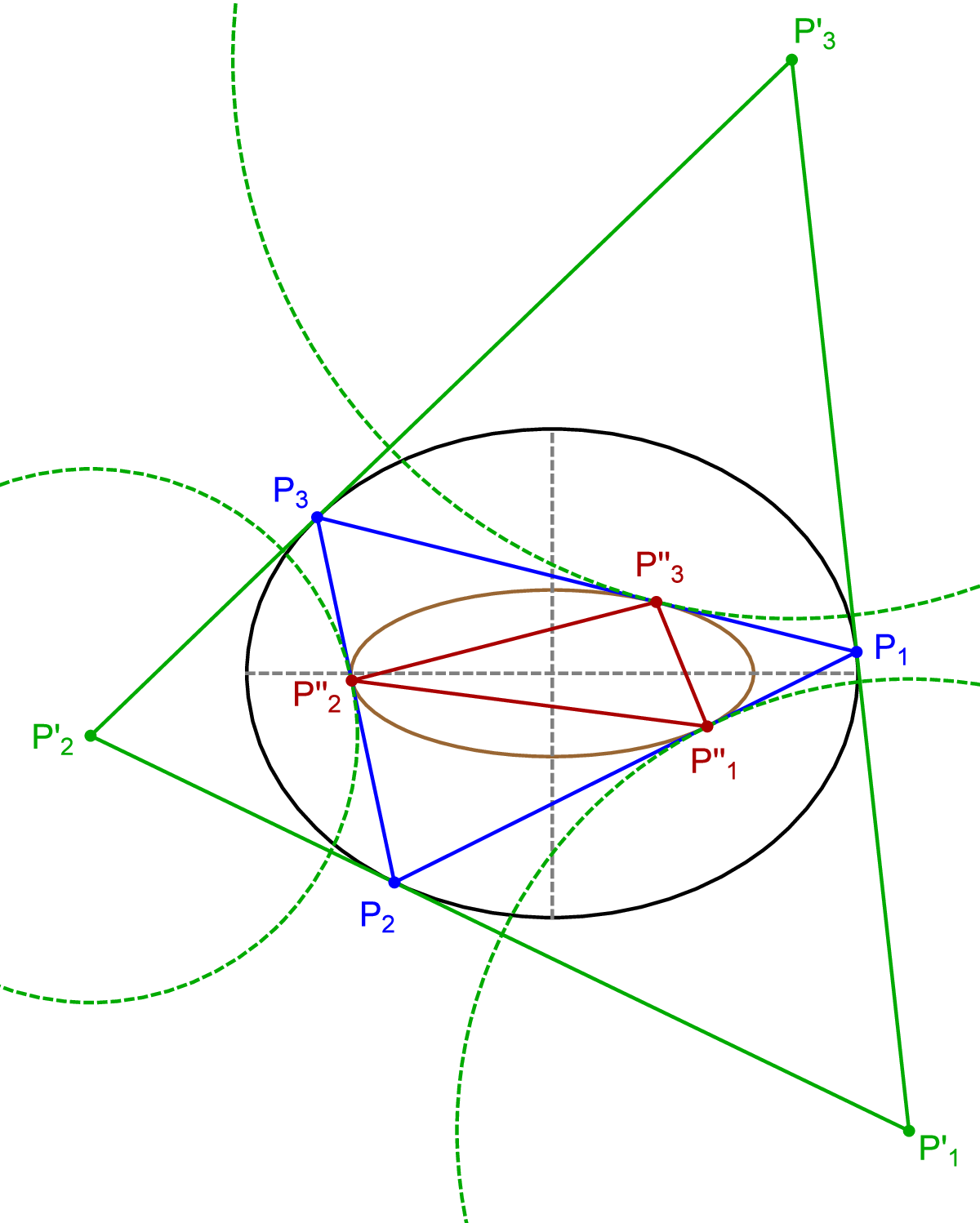}
    \caption{An $a/b=1.25$ elliptic billiard is shown (black) as well as a 3-periodic orbit $T=P_1P_2P_3$ (blue). The orbit's excentral triangle $T'=P_1'P_2'P_3'$ (green) has vertices at the intersections of exterior bisectors (therefore it is tangent to the elliptic billiard at the orbit vertices). The extouch triangle $T''=P_1''P_2''P_3''$ (red) has vertices   where excircles (dashed green) touch each side. Said vertices are known to lie on the orbit's caustic (brown) as it is its Mandart inellipse \cite[Mandart Inellipse]{mw}.  $A,A',A''$ are the areas of $T,T',T''$, respectively \textbf{Video}: \cite[PL\#06]{reznik2020-playlist-proofs}}
    \label{fig:orbit-outer-inner}
\end{figure}

\subsection{Excentral-to-Extouch Area Ratio.}

Let $A''$ denote the area of the extouch triangle, whose vertices are where each excircle touches a side \cite[Extouch Triangle]{mw}; see Figure~\ref{fig:orbit-outer-inner}.

\begin{remark}  
The vertices of the orbit's extouch triangle lie on the caustic.
\end{remark}

Indeed, these are known lie on the Mandart inellipse, whose center is the Mittenpunkt $X_9$ \cite[Mandart Inellipse]{mw}; therefore it must be the caustic; see (\cite[PL\#06]{reznik2020-playlist-proofs}).

\begin{theorem}
$A'/A''$ is invariant and equal to $(2R/r)^2$.
\label{thm:area-ratio-outer-inner}
\end{theorem}

\begin{proof}
Combine the expressions for the area of the excentral \cite[Excentral Triangle]{mw} and extouch triangles \cite[Extouch Triangle]{mw} as a function of sidelengths $s_i$, inradius $r$, and perimeter $L$, obtain $A'/A=A/A''=(s_1 s_2 s_3)/(r^2 L)$. Since $A'/A=r/(2R)$  and $A'/A''=(A'/A)(A/A'')$, the result follows from direct calculations.
\end{proof}

\subsection{Billiard-to-Caustic Area Ratio.}

The $N=3$ caustic semi-axes are given by \cite{garcia2019-ellipses,reznik2020-loci}:

\begin{align}
a_c=&\frac{a\left(\delta-{b}^{2}\right)}{c^2},\;\;\;\;
b_c=\frac{b\left({a}^{2}-\delta\right)}{c^2}\cdot
\end{align}

Note that two concentric, axis-aligned ellipses produce a 3-periodic Poncelet family if and only if $a/a_c+b/b_c=1$ \cite{georgiev2012-poncelet}, which holds for the above $a_c,b_c$.

\begin{corollary}
\label{cor:caustica_billiard}
Let $A_b$ and $A_c$ be the areas of the elliptic billiard and the $N=3$ confocal caustic, respectively. Then

\begin{equation}
\frac{A_c}{A_b}=\frac{r}{2R}\cdot
\label{eqn:caustic-area-ratio}
\end{equation}
\end{corollary}

\begin{proof}
As $A_c=\pi{a_c}{b_c}$ and $A_b={\pi}{a}{b}$, the result follows from equation~\eqref{eqn:rovR}.
\end{proof}

  Notice the caustic-to-billiard area ratio is equal to the orbit-to-excentral area ratio, equation~\eqref{eqn:area-ratio}. Recall also that the caustic semi-axes $a_c,b_c$ depend on $N$, i.e., equation  \eqref{eqn:caustic-area-ratio} is specific for the $N=3$ case.

The power of a point $Q$ with respect to a circle centered at $C_0$ of radius $\mu$ is given by $|Q-C_0|^2-\mu^2$ \cite[Circle Power]{mw}. Let $\mathcal{C}$ be the (moving) circumcircle to the 3-periodic orbits.

\begin{theorem}\label{thm:power_delta}
The power of the billiard center $O$ with respect to $\mathcal{C}$ is invariant and equal to $-\delta$.
\label{thm:delta-power}
\end{theorem}

\begin{proof}
 Consider an isosceles 3-periodic orbit with $P_1=[a,0]$,
 {\small  
 \[  \;  P_2   =\left[\frac {{a}^{2}\sqrt {2\,\delta-{a}^{2}-{b}^{2}}}{{a}^{2}-{b}^{2}},   
	\frac { \left(\delta  -{a}^{2}\right) b}{{a}^{2}-{b}^{2}}\right]\;\;
	 \text{ and  }\;\;  P_3= \left[{\frac {{a}^{2}\sqrt {2\,\delta-{a}^{2}-{b}^{2} }}{{a}^{2}-{b}^{2}}},
	{\frac { \left(  {a}^{2}-\delta \right) b}{{a}^{2}-{b}^{2}}}\right].
	\]
 }%
	Its circumcircle will be centered at $C_0=[ {\frac { {b}^{2}-\delta}{2b}},0]$ with circumradius $R_0=\frac {{b}^{2}+\delta}{2b}.$
	Therefore, the power of the center of the ellipse with respect to the circumcircle is given by  
	$$|OC_0|^2-R_0^2=\left(\frac { {b}^{2}-\delta}{2b}\right)^2 - \left(\frac {{b}^{2}+\delta}{2b}\right)^2=-\delta.$$

For a generic 3-periodic orbit, the stated invariance is confirmed via a CAS, using the explicit vertex expressions in \cite{garcia2019-ellipses} reproduced in Appendix \ref{app:orbit-vertices}.
\end{proof}
%%% END \input{005_invariant_ratio}

\section{Cosine Circle is External to Billiard.}
\label{sec:cosine-circle}
%%% BEGIN \input{006_cosine_circle}
The {\em cosine circle} (also known as the second Lemoine circle) \cite[Cosine Circle]{mw} of a triangle passes through 6 points: the 3 pairs of intersections of sides with lines drawn through the symmedian $X_6$ parallel to sides of the orthic triangle. Recall that the orthic vertices are the feet of altitudes. Its center is $X_6$ \cite[Cosine Circle]{mw}. If one takes the excentral triangle of an orbit as the reference triangle, it is easy to see its orthic is the orbit itself; see Figure~\ref{sec:cosine-circle}.

\begin{theorem}
The cosine circle of the excentral triangle is invariant over the family of 3-periodic orbits. Its radius $r^*$ is constant and it is concentric and external to the elliptic billiard.
\label{thm:cosine-circle}
\end{theorem}

\begin{proof}
Once again, we set $P_1=(a,0)$, derive a candidate expression for $r^*$, and with a CAS check if the claim holds for all $x_1\in(-a,a)$. This yields

\begin{equation}
r^*=\frac{a^2-b^2}{\sqrt{2\delta-a^2-b^2}}\cdot
\end{equation}

Let $a>b>0$ and $\delta=\sqrt{a^4-a^2 b^2+b^4}$. As $0 < {(a^2-b^2)}^{2}<\delta^2$; it follows that
\begin{equation*} (r^{*})^{2}=  \frac{a^2+ b^2+2\delta}{3}
  > \; \frac{a^2+b^2+2(a^2-b^2)}{3} > \; a^2\cdot
\end{equation*}

As the excentral cosine circle and the elliptic billiard are concentric, the proof is complete.
\end{proof}

\begin{figure}[H]
    \centering
    \includegraphics[width=.65\textwidth]{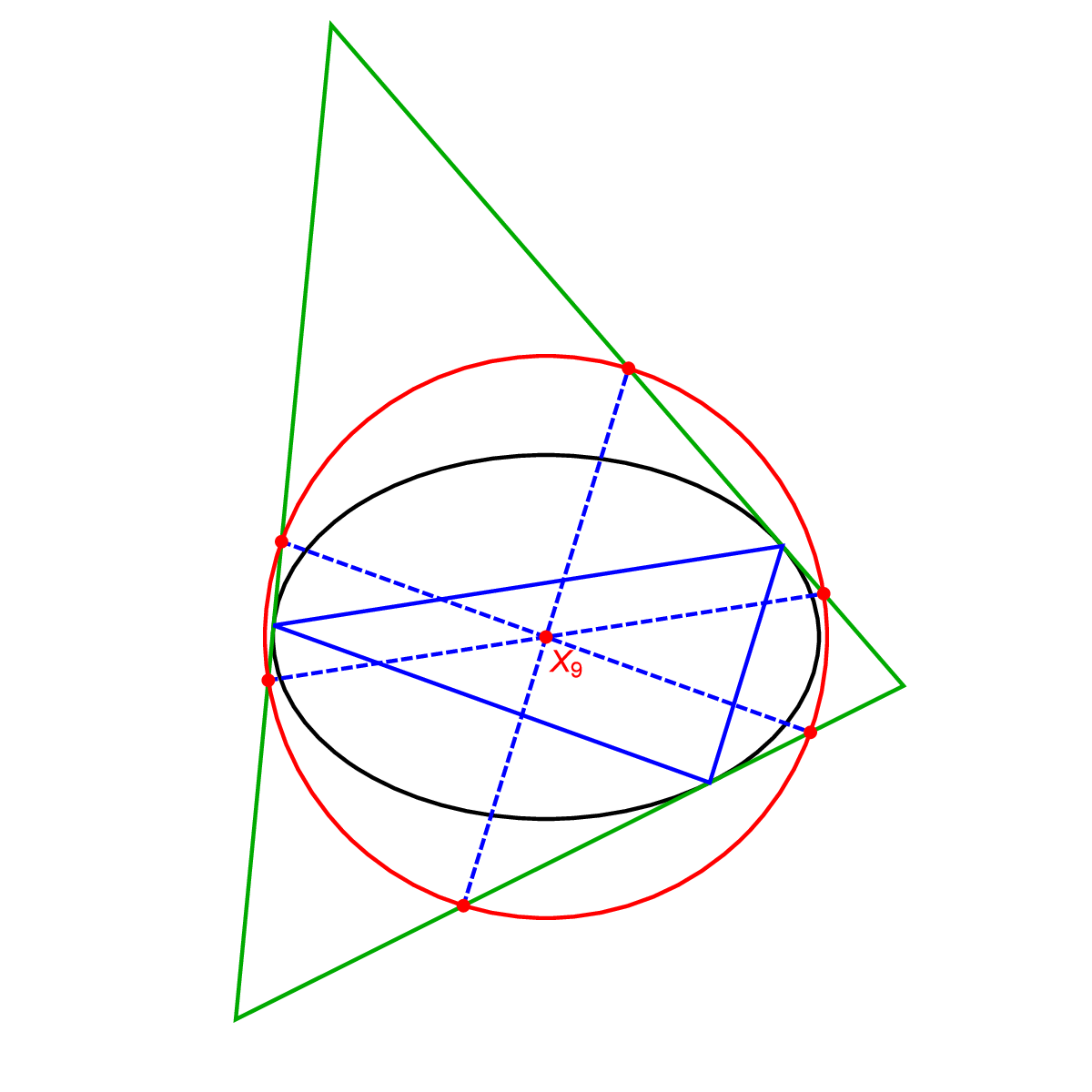}
    \caption{Given a 3-periodic orbit (blue), the cosine circle (red) of the excentral triangle (green) passes through the three pairs of intersections of the dashed lines with the excentral triangle. These are lines parallel to the orbit's sides drawn through the excentral's symmedian point $X_6$, congruent with the orbit's Mittenpunkt $X_9$. This circle is stationary across the 3-periodic orbit family and always exterior to the elliptic billiard. \textbf{Video:} \cite[PL\#04]{reznik2020-playlist-proofs}.}
    \label{fig:cos-circle}
\end{figure}

\begin{remark}
S. Tabachnikov kindly contributed 
\cite{reznik2019-intelligencer} the following equivalent expression for $r^*$, valid for all $N$:

\begin{equation}
    r^* = 1/\gamma\cdot
\label{eqn:rstar-sergei}
\end{equation}
\end{remark}

\begin{remark}[Youtube Math]

In the comments section to our cosine circle video \cite[PL\#05]{reznik2020-playlist-proofs}, D. Laurain contributed an alternative expression for $r^*$ \cite{dominique19}:

\begin{equation}
    r^* = \frac{2 L}{r/R+4}\cdot
    \label{eqn:rstar-laurain}
\end{equation}
\end{remark}

\noindent Combining this with $r^*=1/\gamma$ \eqref{eqn:rstar-sergei}, we obtain the following.

\begin{theorem}
${r/R}={\gamma}L-4$, where $\gamma$ is Joachimsthal's constant and $L$ is the orbit's perimeter.
\label{thm:rovR-explicit}
\end{theorem}
%%% END \input{006_cosine_circle}

\section{Generalizing to $N>3$.}
\label{sec:gener}
%%% BEGIN \input{007_generalizations}
Theorem~\ref{thm:rovR-explicit} motivated us to conjecture that for all $N$

\begin{equation*}
\sum_{i=1}^{N}{\cos\theta_i}=1+\frac{r}{R}=\gamma L - N\cdot
\end{equation*}

\noindent To our delight, this was proved \cite{akopyan2020-invariants,bialy2020-invariants}. Likewise, observing that the excentral-to-orbit area ratio $A'/A$ is invariant for $N=3$ (see equation \eqref{eqn:area-ratio}) motivates querying via experiment whether this quantity remains numerically invariant for $N>3$. Indeed, one observes it does, however only for {\em odd} N. This was also subsequently proved \cite{akopyan2020-invariants}.

Recently, we've numerically detected dozens of new invariants for $N$-periodic orbits \cite{reznik2020-forty-invariants}, many of which have already been   established rigorously. Indeed, curiosity-driven experimentation with the elliptic billiard and related Poncelet families has proven an effective and entertaining approach with which to discover new properties and/or invariants.
%%% END \input{007_generalizations}

\section{Experimental Method.}
\label{sec:experimental}
%%% BEGIN \input{008_experimental}
Our process of discovery of the constancy of $r/R$ started with picking a particular $a/b$ and plotting $R$, $r_9$, and $r$ vs. the $t$ parameter in $P_1(t)$; see Figure~\ref{fig:radii-grid2} (top). This confirmed the well-known relation $R/r_9=2$, valid for any triangle \cite{coxeter67}. It also suggested  that $R$ might be proportional to $r$, which is not true for any triangle family.

To investigate this potential relationship, we produced a scatter plot of orbit triangles for a discrete set of $a/b$ in $(r,R)$-space. One notices that triangles corresponding to individual $a/b$ fall on straight-line segments, and that all segments pass through the origin, suggesting that $r/R$ is a constant; see Figure~\ref{fig:radii-grid2} (bottom). One also notices each segment has endpoints $(R_\text{min},r_\text{min})$ and $(R_\text{max},r_\text{max})$. These are produced at isosceles orbit configurations (see Figure~\ref{fig:rR-min-max}) and are of the form

\begin{align}
r_{\mathrm{min}}=&  \frac{b^2(\delta - b^2)}{c^2 a},\;\;\;R_{\mathrm{min}}= \frac{a^2+\delta}{2a} \nonumber \\
r_{\mathrm{max}}=&  \frac{a^2(a^2 - \delta)}{c^2 b},\;\;\;R_{\mathrm{max}}= \frac{b^2+\delta}{2b}
\label{eqn:rmin-rmax}
\end{align}

\begin{figure}[H]
    \centering
    \includegraphics[width=0.70\textwidth]{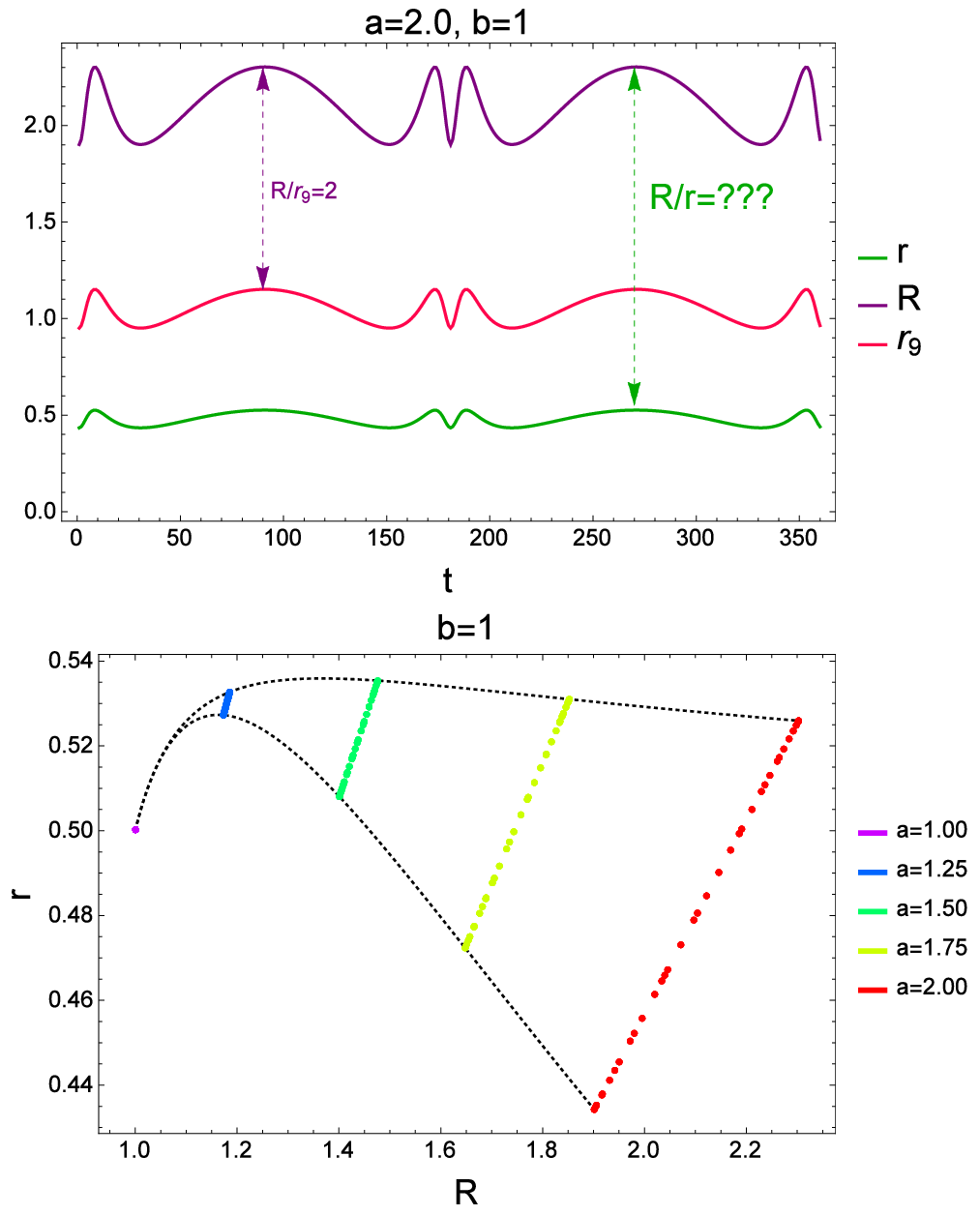}
    \caption{\textbf{Top}: The three radii (inradius $r$, circumradius $R$ and radius of the 9-point circle $r_9$ plotted vs. $t$ of $P_1(t)=(a \cos t,\sin t)$. $R/r_9=2$ holds for any triangle, though $R/r$ is not constant in general. \textbf{Bottom}: Scatter plot in $(r,R)$-space: each dot is a 3-periodic orbit in a discrete set of $a/b$ families. Each family of triangles organizes along a straight line, suggesting the $r/R$ ratio is constant. The minimum and maximum $(r,R)$ that can occur in a family are bound by two continuous curves (dotted black); see \eqref{eqn:rmin-rmax}.}
    \label{fig:radii-grid2}
\end{figure}

For illustration, Figure~\ref{fig:rovR-vs-ab} shows the monotonically decreasing dependence of $r/R$ vs. $a/b$.

\begin{figure}[H]
    \centering
    \includegraphics[width=.60\textwidth]{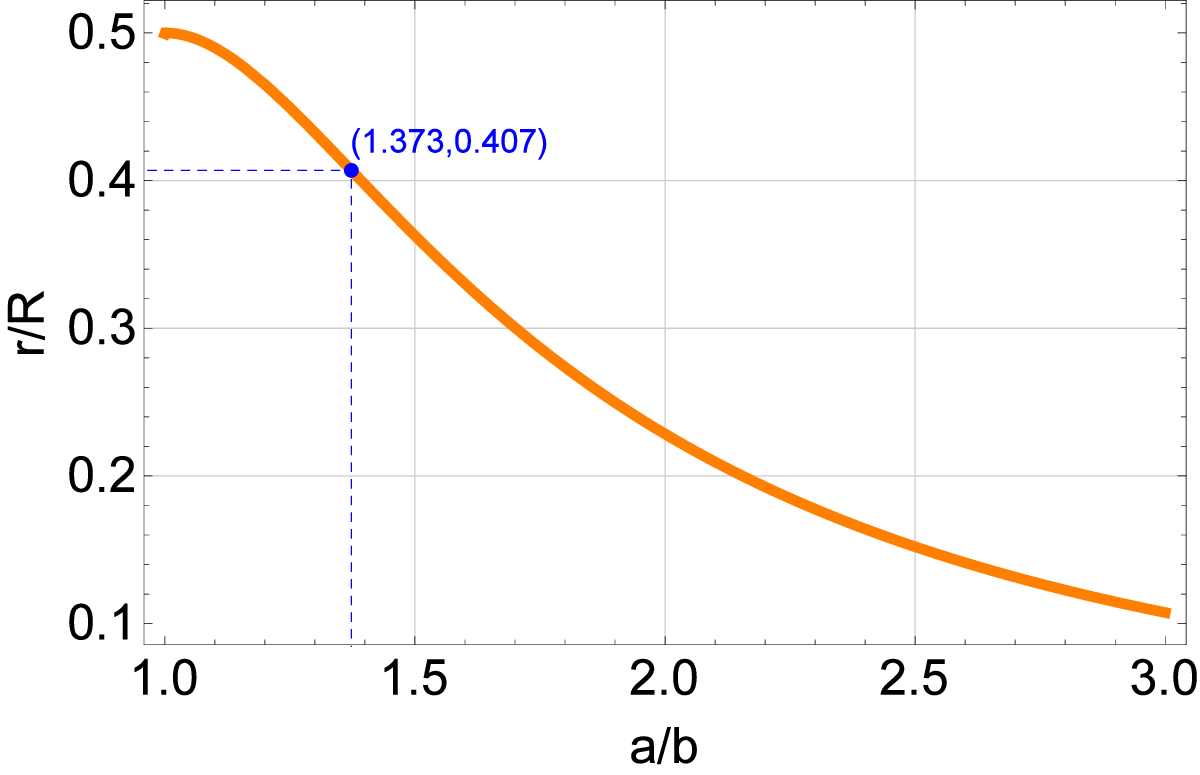}
    \caption{Dependence of $r/R$ on the elliptic billiard aspect ratio $a/b$. The maximum is achieved when the elliptic billiard is a circle. The curve contains an inflection point that still lacks a geometric interpretation.}
    \label{fig:rovR-vs-ab}
\end{figure}
%%% END \input{008_experimental}

\appendix
\section{Video List.} 
\label{app:videos}
%%% BEGIN \input{101_app_videos}
Table~\ref{tab:playlist} lists videos mentioned or illustrating some phenomena in this article. These can be viewed in sequence in the YouTube playlist \cite{reznik2020-playlist-proofs}.

\begin{table}[H]
\caption{Videos mentioned in the article. Column ``PL\#'' indicates the entry within a YouTube playlist \cite{reznik2020-playlist-proofs}.}
\begin{tabular}{c|l|l}
\href{https://bit.ly/2Gmn73e}{PL\#} & Title & Section\\
\hline

\href{https://youtu.be/9xU6T7hQMzs}{01} &
{Locus of incenter and intouchpoint} & \ref{sec:intro} \\

\href{https://youtu.be/tMrBqfRBYik}{02} &
{Mittenpunkt stationary at elliptic billiard center} & \ref{sec:intro} \\

\href{https://youtu.be/P8ykpE_ZbZ8}{03} &
{Constant cosine sum and product} &
\ref{sec:rovr} \\

\href{https://youtu.be/ACinCf-D_Ok}{04} &
\makecell[lt]{Excentral cosine circle is\\stationary and exterior to elliptic billiard} &
\ref{sec:cosine-circle}\\

\href{https://youtu.be/hCQIT6_XhaQ}{05} &
\makecell[lt]{Alternative cosine circle video\\with D. Laurain's $r^*$ formula \eqref{eqn:rstar-laurain}} &
\ref{sec:cosine-circle}\\

\href{https://youtu.be/1gYb5Y3-rQI}{06} & Extouchpoints lie on $N=3$ caustic
 &
\ref{sec:rovr}\\

\end{tabular}
\label{tab:playlist}
\end{table}
%%% END \input{101_app_videos}

\section{Orbit Vertices.}
\label{app:orbit-vertices}
%%% BEGIN \input{103_app_orbit_vertices}
Let the boundary of the billiard satisfy equation~\eqref{eqn:billiard-f}. Without loss of generality, assume that $a {\ge} b$.

Given a starting vertex $P_1$,  exit angle $\alpha$ (see equation~\eqref{eqn:cosalpha}), and a 3-periodic orbit $P_1P_2P_3$, vertex $P_2$ will be given by $(p_{2x},p_{2y})/q_2$, with \cite{garcia2019-ellipses}

\begin{align*}
%\label{eqn:p2}
p_{2x}=&-{b}^{4} \left(  \left(   a^2+{b}^{2}\right)\cos^{2}\alpha   -{a}^{2}  \right) x_1^{3}-2\,{a}^{4}{b}^{2} \cos \alpha  \,\sin  \alpha  \,  x_1^{2}{y_1}\\
&+{a}^{4} \left(  ({a
}^{2}-3\, {b}^{2}) \cos^{2} \alpha  +{b}^{2}
 \right) {x_1}\,y_1^{2}-2{a}^{6} \,\cos  \alpha  \sin   \alpha  \, y_1^{3} ,
\\
p_{2y}=&\; 2{b}^{6} \,\cos \alpha\sin \alpha\,   x_1^{3}+{b}^{4}\left(  ({b
 }^{2}-3\, {a}^{2}) \cos^{2} \alpha  +{a}^{2}
  \right) x_1^{2}{y_1}\\
&+  2\,{a}^{2} {b}^{4}\cos \alpha \sin
  \alpha \; {x_1} y_1^{2} -{
a}^{4}  \left(  \left(   a^2+{b}^{2}\right)\cos^{2}  \alpha  -{b}^{2}  \right)  y_1^{3}
\\
q_2=& \; {b}^{4} \left( a^2-c^2\cos^2\alpha   \right)
x_1^{2}+{a}^{4} \left(  {b}^{2}+c^2\cos^2 \alpha  
 \right) y_1^{2}\\
 & - 2\, {a}^{2}{b}^{2}{c^2}\cos \alpha\sin \alpha \; {x_1}\,{
y_1},
\end{align*}

and vertex $P_3$ will be given by $(p_{3x},p_{3y})/q_3$, with

\begin{align*}
p_{3x}=& \; {b}^{4} \left( {a}^{2}- \left( {b}^{2}+{a}^{2} \right) \right)
 \cos^{2}  \alpha   x_1^{3} +2\,{a}^{4}{b}^{2} \cos  \alpha \sin \alpha\,   x_1^{2}{ y_1}\\
 &+{a}^{4} \left( 
  \cos^{2}  \alpha  \left( {a}^{2}-3\,{b}^{2}
 \right) +{b}^{2} \right) { x_1}\, y_1^{2} +2\, {a}^{6} 
 \cos \alpha\,\sin \alpha\, y_1^{3}
\\
p_{3y}=&  -2\, {b}^{6} \cos \alpha\sin \alpha\, x_1^{3}+{b}^{4} \left( {a}^{2}+ \left( {b}^{2}-3\,{a}^{2} \right)    \cos^2\alpha \right) {{ x_1}}^{2}{ y_1}
\\
& -2\,{a}^{2}  {b}^{4}\cos
 \alpha  \sin \alpha\,  x_1 y_1^{2}+
 {a}^{4} \left( {b}^{2}- \left( {b}^{2}+{a}^{2} \right)   \cos^{2}  \alpha  \right)\,  y_1^{3},
\\
q_3=& \;{b}^{4} \left( {a}^{2}-{c^2}\cos^{2}\alpha   \right) x_1^{2}+{a}^{4} \left( {b}^{2}+c^2\cos^{2}\alpha  \right)  y_1^{2}\\
&+2\,{a}^{2}{b}^{
2} c^2 \cos\alpha \sin\alpha\, { x_1}\,{ y_1}.
\end{align*}
%%% END \input{103_app_orbit_vertices}

%%% BEGIN \input{015_ack}
\begin{acknowledgment}{Acknowledgments.}
 We would like to thank  
Sergei Tabachnikov, Richard Schwartz, Arseniy Akopyan, Olga Romaskevich, Corentin Fierobe, Ethan Cotterill, Dominique Laurain, and Mark Helman for generously helping us during this research.  The first author is fellow of CNPq and coordinator of Project PRONEX, CNPq, FAPEG 2017 10 26 7000 508. The third author thanks the Federal University of Juiz de Fora for a 2018-2019 fellowship.
\end{acknowledgment}
%%% END \input{015_ack}

%%% BEGIN \input{030_biblio}

%%% END \input{030_biblio}

%%% BEGIN \input{020_bios}
\begin{biog}

\item[Ronaldo Garcia] is a professor at the Federal University of Goiás, Brazil. He holds a BSc. in electrical engineering and received his doctorate in mathematics from IMPA, Rio de Janeiro, Brazil, under Jorge Sotomayor. His research interests include classical differential geometry, differential equations, and dynamical systems, in particular the qualitative theory of curvature lines and others geometric singular foliations of surfaces and hypersurfaces of Euclidean spaces.
\begin{affil}
Instituto de Matemática e Estatística\\
Universidade Federal de Goiás\\
Goiânia, Brazil, 74690-900\\
{ragarcia@ufg.br}
\end{affil}

\item[Dan Reznik] is director at Data Science Consulting, Rio de Janeiro, Brazil. He holds a Ph.D. in computer science from the University of California, Berkeley, having worked in such fields as automation, robotics, solar Energy, and advanced data analytics.
\begin{affil}
Data Science Consulting\\
Rio de Janeiro, Brazil\\
{dreznik@gmail.com}
\end{affil}

\item[Jair Koiller] was Visiting Professor at the Mathematics Department of the Federal University of Juiz de Fora and retired professor in the Mathematics Institute at Federal University of Rio de Janeiro, Brazil. His research is on applied tools of differential geometry and dynamical Systems. He has collaborated in such areas as astronomy, physics, engineering and biology. His recent interests include applied social sciences and education, especially outreach science projects.

\begin{affil}
Universidade Federal do Rio de Janeiro\\
Rio de Janeiro, Brazil\\
{jairkoiller@gmail.com}
\end{affil}
\end{biog}
%%% END \input{020_bios}

\begin{thebibliography}{10}
\expandafter\ifx\csname urlstyle\endcsname\relax
 \providecommand{\url}[1]{doi:\discretionary{}{}{}#1}\else
 \providecommand{\url}{doi:\discretionary{}{}{}\begingroup
  \urlstyle{rm}\Url}\fi

\bibitem{akopyan2020-invariants}
Akopyan, A., Schwartz, R., Tabachnikov, S. (2020).
\newblock Billiards in ellipses revisited.
\newblock \emph{Eur. J. Math.}
\newblock \url{doi.org/10.1007/s40879-020-00426-9}.

\bibitem{bialy2020-invariants}
Bialy, M., Tabachnikov, S. (2020).
\newblock {Dan Reznik's} identities and more.
\newblock \emph{Eur. J. Math.}
\newblock \url{doi.org/10.1007/s40879-020-00428-7}.

\bibitem{ana2020}
Chavez-Caliz, A. (2020).
\newblock More about areas and centers of {Poncelet} polygons.
\newblock \emph{Arnold Math J.}
\newblock \url{doi.org/10.1007/s40598-020-00154-8}.

\bibitem{coxeter67}
Coxeter, H. S.~M., Greitzer, S.~L. (1967).
\newblock \emph{Geometry Revisited}, vol.~19 of \emph{New Mathematical
  Library}.
\newblock New York: Random House, Inc.

\bibitem{dragovic11}
Dragovi\'{c}, V., Radnovi\'{c}, M. (2011).
\newblock \emph{{Poncelet} Porisms and Beyond: Integrable Billiards,
  Hyperelliptic Jacobians and Pencils of Quadrics}.
\newblock Frontiers in Mathematics. Basel: Springer.

\bibitem{dragovic88}
Dragovi\'{c}, V., Radnovi\'{c}, M. (2019).
\newblock Caustics of {P}oncelet polygons and classical extremal polynomials.
\newblock \emph{Regul. Chaotic Dyn.}, 24(1): 1--35.

\bibitem{corentin2021-circum}
Fierobe, C. (2021).
\newblock On the circumcenters of triangular orbits in elliptic billiard.
\newblock \emph{Journal of Dynamical and Control Systems}, 27(4): 693--705.

\bibitem{Gallatly1913}
Gallatly, W. (1913).
\newblock \emph{The Modern Geometry of the Triangle}.
\newblock London: Hodgson.

\bibitem{garcia2019-ellipses}
Garcia, R. (2019).
\newblock Elliptic billiards and ellipses associated to the 3-periodic orbits.
\newblock \emph{Amer. Math. Monthly}, 126(06): 491--504.
\newblock \url{doi.org/10.1080/00029890.2019.1593087}.

\bibitem{garcia2020-poristics}
Garcia, R., Reznik, D. (2020).
\newblock Related by similiarity: Poristic triangles and 3-periodics in the
  elliptic billiard.
\newblock \url{arxiv.org/abs/2004.13509}.

\bibitem{georgiev2012-poncelet}
Georgiev, V., Nedyalkova, V. (2012).
\newblock {Poncelet}’s porism and periodic triangles in ellipse.
\newblock \emph{Dynamat}.
\newblock \url{bit.ly/2G33bFQ}.

\bibitem{johnson29}
Johnson, R.~A. (1929).
\newblock \emph{Modern Geometry: An Elementary Treatise on the Geometry of the
  Triangle and the Circle}.
\newblock Boston, MA: Houghton Mifflin.

\bibitem{kimberling97-major-centers}
Kimberling, C. (1997).
\newblock Major centers of triangles.
\newblock \emph{Amer. Math. Monthly}, 104(5): 431--438.

\bibitem{etc}
Kimberling, C. (2019).
\newblock Encyclopedia of triangle centers.
\newblock \emph{ETC}.
\newblock \url{bit.ly/3mOOver}.

\bibitem{Nikolina-families2012}
Kova\v{c}evi\'c, N., Sliep\v{c}evi\'c, A. (2012).
\newblock On the certain families of triangles.
\newblock \emph{KoG-Zagreb}, 16: 21--27.

\bibitem{dominique19}
Laurain, D. (2019).
\newblock Formula for the radius of the orbits' excentral cosine circle.
\newblock Private communication.

\bibitem{sergei07_grid}
Levi, M., Tabachnikov, S. (2007).
\newblock The {P}oncelet grid and billiards in ellipses.
\newblock \emph{Amer. Math. Monthly}, 114(10): 895--908.

\bibitem{Murnaghan1925}
Murnaghan, F.~D. (1925).
\newblock Discussions: Note on {Mr. Weaver}'s paper ``a system of triangles
  related to a poristic system'' (1924, 337--340).
\newblock \emph{Amer. Math. Monthly}, 32(1): 37--41.
\newblock \url{www.jstor.org/stable/2300090}.

\bibitem{Odenhal2011}
Odehnal, B. (2011).
\newblock Poristic loci of triangle centers.
\newblock \emph{Journal for Geometry and Graphics}, 15(1): 45--67.

\bibitem{Pamfilos2011}
Pamfilos, P. (2011).
\newblock Triangles with given incircle and centroid.
\newblock \emph{Forum Geometricorum}, 11: 27--51.

\bibitem{reznik2019-locus-gallery}
Reznik, D. (2019).
\newblock Triangular orbits in elliptic billards: Loci of points {X(1)~X(100)}.
\newblock \emph{GitHub}.
\newblock \url{dan-reznik.github.io/billiard-loci/loci.html}.

\bibitem{reznik2020-playlist-proofs}
Reznik, D. (2020).
\newblock Playlist for {``New Properties of Triangular Orbits in an Elliptic
  Billiard''}.
\newblock \emph{YouTube}.
\newblock \url{bit.ly/379mk1I}.

\bibitem{reznik2019-intelligencer}
Reznik, D., Garcia, R., Koiller, J. (2020).
\newblock Can the elliptic billiard still surprise us?
\newblock \emph{Math. Intelligencer}, 42: 6--17.
\newblock \url{rdcu.be/b2cg1}.

\bibitem{reznik2020-forty-invariants}
Reznik, D., Garcia, R., Koiller, J. (2020).
\newblock Forty new invariants of {N}-periodics in the elliptic billiard.
\newblock \emph{arXiv}.
\newblock \url{arxiv.org/abs/2004.12497}.

\bibitem{reznik2020-loci}
Reznik, D., Garcia, R., Koiller, J. (2020).
\newblock Loci of 3-periodics in an elliptic billiard: why so many ellipses?
\newblock \url{arxiv.org/abs/2001.08041}.

\bibitem{olga14}
Romaskevich, O. (2014).
\newblock On the incenters of triangular orbits on elliptic billiards.
\newblock \emph{Enseign. Math.}, 60(3-4): 247--255.

\bibitem{rozikov2018}
Rozikov, U.~A. (2018).
\newblock \emph{An Introduction To Mathematical Billiards}.
\newblock Singapore: World Scientific Publishing Company.

\bibitem{schwartz2018-rectangles}
Schwartz, R. (2019).
\newblock Rectangle coincidences and sweepouts.
\newblock \url{arxiv.org/abs/1809.03070}.

\bibitem{schwartz2016-com}
Schwartz, R., Tabachnikov, S. (2016).
\newblock Centers of mass of {P}oncelet polygons, 200 years after.
\newblock \emph{Math. Intelligencer}, 38(2): 29--34.
\newblock \url{bit.ly/3kLetNU}.

\bibitem{Sliepcevic2013}
Sliep\v{c}evi\'c, A., Halas, H. (2013).
\newblock Family of triangles and related curves.
\newblock \emph{Hrvat. Akad. Znan. Umjet. Mat. Znan}, 17(515): 203--2010.

\bibitem{sergei91}
Tabachnikov, S. (2005).
\newblock \emph{Geometry and Billiards}, vol.~30 of \emph{Student Mathematical
  Library}.
\newblock Providence, RI: American Mathematical Society.
\newblock \url{bit.ly/2RV04CK}.

\bibitem{Weaver1924}
Weaver, J.~H. (1924).
\newblock A system of triangles related to a poristic system.
\newblock \emph{Amer. Math. Monthly}, 31(7): 337--340.
\newblock \url{www.jstor.org/stable/2299387}.

\bibitem{Weaver1933}
Weaver, J.~H. (1933).
\newblock Curves determined by a one-parameter family of triangles.
\newblock \emph{Amer. Math. Monthly}, 40(2): 85--91.
\newblock \url{www.jstor.org/stable/2300940}.

\bibitem{mw}
Weisstein, E. (2019).
\newblock Mathworld.
\newblock \emph{MathWorld--A Wolfram Web Resource}.
\newblock \url{mathworld.wolfram.com}.

\end{thebibliography}
\end{document}